\documentclass[12pt]{amsart}
\usepackage{latexsym,amscd, amsmath, amssymb,epsfig,xypic,tikz} 

\theoremstyle{plain}
  \newtheorem{theorem}{Theorem}[section]
  
  \newtheorem{lemma}[theorem]{Lemma}
  \newtheorem{corollary}[theorem]{Corollary}
  
\theoremstyle{definition}

  \newtheorem{question}[theorem]{Question}
  
\theoremstyle{remark}

\numberwithin{equation}{section}

\def\umapright#1{\smash{
   \mathop{\longrightarrow}\limits^{#1}}}

\def\rmapdown#1{\Big\downarrow\rlap
   {$\vcenter{\hbox{$\scriptstyle#1$}}$}}

\def\tempbaselines
{\baselineskip22pt\lineskip3pt
   \lineskiplimit3pt}
\def\diagram#1{\null\,\vcenter{\tempbaselines
\mathsurround=0pt
    \ialign{\hfil$##$\hfil&&\quad\hfil$##$\hfil\crcr
      \mathstrut\crcr\noalign{\kern-\baselineskip}
  #1\crcr\mathstrut\crcr\noalign{\kern-\baselineskip}}}\,}

\def\pullback#1&#2&#3&#4&#5&#6&#7&#8&{
\diagram{#1&\umapright{#2}&#3\cr
\rmapdown{#4}&&\rmapdown{#5}\cr
#6&\umapright{#7}&#8\cr}}

\def\calC{{\mathcal C}}

\def \End{\mathop{\rm End}\nolimits}

\def \Hom{\mathop{\rm Hom}\nolimits}

\def \Rad{\mathop{\rm Rad}\nolimits}

\def\QQ{{\Bbb Q}}

\def\ZZ{{\Bbb Z}}

\begin{document}

\title[Tree Class of  the Auslander-Reiten Quiver]
{Combinatorial Restrictions on the Tree Class of  the Auslander-Reiten Quiver of a Triangulated Category}

\author{Kosmas Diveris}\email{diveris@stolaf.edu}\address{Department of Mathematics\\
St. Olaf College\\ Northfield, MN 55057, USA} 
\author{Marju Purin}\email{purin@stolaf.edu}\address{Department of Mathematics\\
St. Olaf College\\ Northfield, MN 55057, USA}
\author{Peter Webb}
\email{webb@math.umn.edu}
\address{School of Mathematics\\
University of Minnesota\\
Minneapolis, MN 55455, USA}

\subjclass[2000]{}

\keywords{Triangulated category, additive function, Auslander-Reiten quiver, irreducible morphism}

\begin{abstract}
We show that if a connected, Hom-finite, Krull-Schmidt triangulated category has an Auslander-Reiten quiver component with Dynkin tree class then the category has Auslander-Reiten triangles and that component is the entire quiver. This is an analogue for triangulated categories of a theorem of Auslander, and extends a previous result of Scherotzke. We also show that if there is a quiver component with extended Dynkin tree class, then other components  must also have extended Dynkin class or one of a small set of infinite trees, provided there is a non-zero homomorphism between the components. The proofs use the theory of additive functions.
\end{abstract}

\maketitle

\section{Main results}

Let $k$ be a field and let $\calC$ be a $k$-linear triangulated category which is Hom-finite, Krull-Schmidt and connected.
The Auslander-Reiten quiver of $\calC$ is the graph whose vertices are the indecomposable objects of $\calC$ (up to isomorphism) and where we draw an arrow from $X$ to $Y$, labelled with certain multiplicity information, if there is an irreducible morphism from $X$ to $Y$.  We will be concerned with parts of the Auslander-Reiten quiver where  Auslander-Reiten triangles exist. If $U\to V\to W\to U[1]$ is an Auslander-Reiten triangle we write $U=\tau W$ and $W=\tau^{-1}U$ to define the \textit{Auslander-Reiten translate} $\tau$. 

By a \textit{stable component} $\Gamma$ of the Auslander-Reiten quiver we mean a subgraph with the properties: 
\begin{enumerate}
\item for every indecomposable object $M\in\Gamma$, for every $n\in\ZZ$, $\tau^nM$ also lies in $\Gamma$ and,
\item every irreducible morphism beginning or ending at $M$ lies in $\Gamma$.
\end{enumerate}
A stable component $\Gamma$ has the form $\ZZ T/G$ where $T$ is a labelled tree (the \textit{tree class} of $\Gamma$) and $G$ is a group, by \cite{Rie} and \cite{XZ2}. Note that we do not suppose that $\Gamma$ is closed under the shift operation of $\calC$, and also that a stable component is actually a component of the Auslander-Reiten quiver, not a subset of a component with some objects removed.

We say that $\calC$ is \textit{locally finite} if it is Hom-finite and for each indecomposable object $X$ there are only finitely many isomorphism classes of indecomposable objects $Y$ for which $\Hom_\calC(X,Y)\ne 0$ or $\Hom_\calC(Y,X)\ne 0$. This is a strong  condition which implies that Auslander-Reiten triangles do always exist, there is a single Auslander-Reiten quiver component, and it has Dynkin tree class. This was proved by Xiao and Zhu \cite{XZ1, XZ2}. Under the hypothesis of local finiteness, such categories were (partially) classified by Amiot \cite{Ami}.

In our first main result it is not a hypothesis that all Auslander-Reiten triangles exist throughout the entire category $\calC$.

\begin{theorem}
\label{Dynkin-theorem}
Let $\calC$ be a Hom-finite,  Krull-Schmidt, connected triangulated category,
let $\Gamma$ be a stable component of the Auslander-Reiten quiver of $\calC$ and suppose that $\Gamma$ has Dynkin tree class. Then  $\Gamma$ contains every indecomposable object of $\calC$. Furthermore $\calC$ is locally finite. It follows that $\calC$ has Auslander-Reiten triangles and has only finitely many shift-classes of indecomposable objects.
\end{theorem}

In the particular case of bounded derived categories of algebras our result was proved already by Scherotzke \cite{Sch}  using quite different methods: she related irreducible morphisms in the derived category to irreducible morphisms in the category of complexes and chain maps. Our own approach uses only the theory of additive functions. It applies more generally and has the merit of being brief. 

The theorem is a version for triangulated categories of the theorem of Auslander  for modules over an indecomposable finite dimensional algebra, which states that if the Auslander-Reiten quiver of the module category has a finite component then it is the entire quiver. For a triangulated category the condition that there is a finite stable component $\Gamma$ is strong and does imply that the component must have Dynkin tree class. We state this and give the short proof.

\begin{corollary}
Let $\calC$ be a Hom-finite,  Krull-Schmidt, connected triangulated category with a finite stable component $\Gamma$ of its Auslander-Reiten quiver. Then the tree class of $\Gamma$ is a Dynkin diagram and $\Gamma$ contains every indecomposable object of $\calC$.
\end{corollary}

\begin{proof}
We may use the argument of \cite[Theorem 2.3.5]{XZ2}: the function 
$$
f(X):=\sum_{M\in\Gamma}\dim\Hom_\calC(M,X)
$$
is subadditive and periodic on $\Gamma$, and this forces $\Gamma$ to have Dynkin tree class by \cite{HPR}. The result now follows by Theorem~\ref{Dynkin-theorem}.
\end{proof}

 We also have a theorem about Auslander-Reiten quiver components which have extended Dynkin tree class.
 
 \begin{theorem}
\label{extended-Dynkin-theorem}
Let $\calC$ be a Hom-finite,  Krull-Schmidt, connected triangulated category with stable Auslander-Reiten quiver components $\Gamma_1$ and $\Gamma_2$. Suppose there are objects $X\in\Gamma_1$ and $Y\in\Gamma_2$ so that either $\Hom(X,Y)\ne0$ or $\Hom(Y,X)\ne0$. If $\Gamma_1$ has an extended Dynkin diagram as its tree class then the tree class of $\Gamma_2$ is either an extended Dynkin diagram or one of the trees $A_\infty$, $B_\infty$, $C_\infty$, $D_\infty$ or $A_\infty^\infty$.
\end{theorem}

The above trees are displayed in \cite{DR, HPR, Web1}, for instance. Although it is remarkable that the shape of one quiver component influences the shape of other components in this way, the theorem still allows for many possibilities and we wonder if they can all occur in triangulated categories. We ask the following: 

\begin{question}
With the hypotheses and notation of Theorem~\ref{extended-Dynkin-theorem}, is it possible to have  stable Auslander-Reiten quiver components $\Gamma_1$ and $\Gamma_2$ with $\Hom(X,Y)\ne 0$ for some $X\in \Gamma_1$ and $Y\in \Gamma_2$, so that $\Gamma_1$ and $\Gamma_2$ have  different (finite) extended Dynkin diagrams as their tree classes? Is it even the case that $\Gamma_2$ must  either have the same tree class as $\Gamma_1$ or else have tree class $A_\infty$? 
\end{question}

We mention the definitions of some of the standard terms we have been using. We say that $\calC$ is \textit{Hom-finite} if for every pair of indecomposable objects $X$ and $Y$ in $\calC$, $\dim_k\Hom_\calC(X,Y)$ is finite; we say that $\calC$ is \textit{Krull-Schmidt} if every object in $\calC$ is a finite direct sum of indecomposable objects each of which has a local endomorphism ring; and we say that $\calC$ is \textit{connected} if it is not possible to partition the indecomposable objects of $\calC$ into two classes without non-zero homomorphisms between objects in the different classes.
For an introduction to Auslander-Reiten theory in the setting of triangulated categories see \cite{Hap}.  Note that, according to \cite{RV}, the existence of Auslander-Reiten triangles in $\calC$ is equivalent to the existence of a Serre functor on $\calC$.

\section{Additive functions}

Let $\Gamma$ be a stable component of the Auslander-Reiten quiver of $\calC$. We will say that a function $\phi: \Gamma_0\to\ZZ$ is \textit{additive} if on each Auslander-Reiten triangle $U\to (V_1\oplus\cdots\oplus V_r)\to W\to U[1]$ whose first three terms lie in $\Gamma_0$ we have
$\phi(U)+\phi(W)=\phi(V_1)+\cdots+\phi(V_r)$. We will refer to these first three terms as a \textit{mesh} of the quiver. We say that $\phi$ is \textit{positive} if it takes non-negative values, and somewhere is positive. We say that $\phi$  is \textit{defective} on this mesh \textit{of defect $d>0$}
if $\phi(U)+\phi(W)=\phi(V_1)+\cdots+\phi(V_r)+d$. Because $\Gamma\cong \ZZ T/G$ is a quotient of $\ZZ T$, $\phi$ lifts to a function on the vertices of $\ZZ T$, which is additive or defective according as $\phi$ is on $\Gamma$.

When $T$ is a finite tree it defines a Coxeter group which acts on a vector space $\QQ^{T_0}$ with basis indexed by the vertices of $T$. A \textit{slice} of $\ZZ T$ is a connected subgraph whose vertices are a set of representatives for the $\tau$-orbits in $\Gamma$. Such an $S$ has the same underlying graph as $T$.  When $S$ is a slice of $\ZZ T$ and $f:(\ZZ T)_0\to \ZZ$ is a function, the values $\{f(x)\bigm| x\in S_0\}$ may be regarded as the coordinates of a vector in $\QQ^{T_0}$, which we denote $f(S)$.

The next result is completely standard. Part (3) is a strengthening of a result which appears as  \cite[Cor. 2.4]{Web1}.

\begin{lemma}
\label{additive-function-existence-lemma}
Let $T$ be a finite tree.
\begin{enumerate}
\item If $S$ is a slice of $\ZZ T$ and $f$ is an additive function then $f(\tau S)=cf(S)$, where $c$ is a Coxeter transformation.
\item If $T$ is a Dynkin tree there is no positive additive function on $\ZZ T$.
\item If $T$ is an extended Dynkin diagram then positive additive functions on $\ZZ T$ are periodic, of bounded period (depending on the diagram).
\end{enumerate}
\end{lemma}

\begin{proof}(1) This is well-known; a possible reference is \cite[Lemma 2.3]{Web1}.

(2) This is immediate from \cite[Lemma 1.7]{DR}.

(3) We deduce this result also from \cite[Lemma 1.7]{DR}. Since $T$ is an extended Dynkin diagram  with vertex set $T_0$ there is a null-root $\mathbf{n}\in \QQ^{T_0}$ which is fixed by the Weyl group $W$ and which spans a space $N$, so that the action of $W$ on $\QQ^{T_0}/N$ is as a finite group $\overline W$. Write $m$ for the order of the image $\overline c$ in $\overline W$ of a Coxeter transformation $c\in W$. It is shown in the proof of \cite[Lemma 1.7]{DR} that if $x\in \QQ^{T_0}$ is a vector for which $c^t x$ never has negative components, then $c^mx=x$. From this it follows that the positive additive function $f$ is periodic, with period dividing $m$.
\end{proof}

For each object $M$ of $\calC$ we let $f_M$ and $g_M$ be the functions defined on objects $N$ of $\calC$ by $f_M(N)=\dim\Hom_\calC(M,N)$ and $g_M(N)=\dim\Hom_\calC(N,M)$. It is the fact that these functions are additive on almost all of $\calC$ which is the new ingredient in our approach compared to previous use of additive functions.

\begin{lemma}
\label{f-and-g-lemma}
Let $\Gamma$ be a stable Auslander-Reiten quiver component and let $M$ be an indecomposable object. Let $d:=\dim \End(M)/\Rad\End(M)$.
\begin{enumerate}
\item
$f_M$ is additive on every mesh of $\Gamma$ except the meshes with $M$ and $M[1]$ as the right hand object. 
\item $g_M$ is additive on every mesh of $\Gamma$ except the meshes with $M$ and $M[-1]$ as the left hand object. 
\end{enumerate}
On the excluded meshes the functions are defective, of defect $d$ if $M\not\cong M[1]$ and of defect $2d$ if $M\cong M[1]$. 
\end{lemma}

\begin{proof}
This follows from \cite[Proposition 3.1]{Web2}. We present the argument for $f_M$ here. The argument for $g_M$ is dual. 

Let $U\to V\to W\to U[1]$ be an Auslander-Reiten triangle. In the long exact sequence
$$
\begin{aligned}
\cdots&\xrightarrow{\beta_{-1}}\Hom(M,W[-1])\cr
\xrightarrow{\delta_{-1}}\Hom(M,U)\xrightarrow{\alpha_0}\Hom(M,V)&\xrightarrow{\beta_0}
\Hom(M,W)\cr
\xrightarrow{\delta_0}\Hom(M,U[1])\xrightarrow{\alpha_1}\rlap{$\cdots$}\phantom{\Hom(M,V)}&\phantom{\xrightarrow{\beta_0}
\Hom(M,W)}\cr
\end{aligned}
$$
the Auslander-Reiten lifting property of the triangle shows that $\beta_0$ is epi and hence $\delta_0=0$ except when $M\cong W$, and $\beta_{-1}$ is epi and hence $\delta_{-1}=0$ except when $M\cong W[-1]$. Thus, unless $W\cong M$ or $M[1]$, the middle sequence is a short exact sequence, showing that $f_M$ is additive on this mesh in this case. When $W\cong M$ but $M\not\cong M[1]$ then $M\not\cong W[-1]$ and so the only $\delta$ which is non-zero is $\delta_0$, which has rank $d$. This shows that $f_M$ has defect $d$ on this mesh. The argument is similar when $M\cong W[-1]$. When $M\cong M[1]\cong W$ both $\delta_{-1}$ and $\delta_0$ have rank $d$, so that $f_M$ has defect $2d$ on this mesh.
\end{proof}

\section{Proof of Theorem \ref{Dynkin-theorem}}

The proof of Theorem~\ref{Dynkin-theorem} proceeds in three steps. We suppose that $\Gamma$ is a stable component with tree class $T$, so that $\Gamma\cong \ZZ T/G$ for some group $G$.

Step 1: We show that every object in $\calC$ is the shift of an object in $\Gamma$:
$$
\calC=\bigcup_{i\in\ZZ} \Gamma[i].
$$
For, if not, there is an indecomposable object $M\not\in \bigcup_{i\in\ZZ} \Gamma[i]$ with either $f_M\ne0$ or $g_M\ne 0$ on $\Gamma$. Let us say $f_M\ne 0$, as the argument is similar if $g_M\ne 0$. Then $f_M$ is positive and additive on $\Gamma$ and hence gives a positive additive function on $\ZZ T$, which is not possible if $T$ is Dynkin by Lemma~\ref{additive-function-existence-lemma}(2). This impossibility establishes the assertion.

Step 2: We show that when $\Gamma$ is finite it is closed under shift, and hence $\calC$ is locally finite. To prove this, suppose that $\Gamma$ is finite. If there are only finitely many distinct shifts $\Gamma[i]$ then, by Step 1, $\calC$ has only finitely many indecomposable objects and so is locally finite. From this it follows that $\calC$ has only one component, by~\cite[Theorem 2.3.5]{XZ1}. Otherwise, if there are infinitely many distinct $\Gamma[i]$, it follows that all the $\Gamma[i]$ must be distinct. Let $M$ be an indecomposable object, so that $M\in\Gamma[i]$ for some $i$. By Lemma~\ref{f-and-g-lemma} $f_M$  only fails to be additive on $\Gamma[i]$ and $\Gamma[i+1]$, so by Lemma~\ref{additive-function-existence-lemma} it is only non-zero on these components, and similarly $g_M$ is only non-zero on $\Gamma[i]$ and $\Gamma[i-1]$. This shows that $\calC$ is locally finite, and so there is only one quiver component, again by~\cite[Theorem 2.3.5]{XZ1}.

Step 3: 
We show that when $\Gamma$ is infinite it is closed under shift and locally finite. Assuming $\Gamma$ is infinite we must have $\Gamma\cong \ZZ T$, because factoring out any non-identity admissible group of automorphisms of $\ZZ T$ gives a finite quiver (see, for instance, \cite[Sec. 3]{XZ2}). Let $M$ be any object in $\Gamma$. Then $f_M$ is non-zero on $\Gamma$ (because $f_M(M)\ne 0$), and $f_M$ is additive everywhere on $\Gamma$ except on the mesh whose right hand term is $M$, and also on the mesh whose right hand term is $M[1]$ (if it happens to lie in $\Gamma$). To the left and right of these meshes $f_M$ is periodic, because the Coxeter transformation has finite order. From each of these periodic regions we may extend $f_M$ to a periodic, non-negative additive function on the whole of $\Gamma$, which must be zero. Thus $f_M$ is zero both to the left and to the right of $M$ and $M[1]$. At this point we may deduce that $M[1]$ does lie in $\Gamma$ because otherwise we would deduce that $f_M$ must be 0 at $M$, and also by the same reasoning that $M\ne M[1]$. We conclude that $\Gamma[1]=\Gamma$, so $\Gamma$ is closed under shift. 

Furthermore we have seen that $f_M$ is non-zero only on objects which lie between $M$ and $M[1]$, which is a finite set of indecomposable objects. Similarly $g_M$ is non-zero  only on a finite set of indecomposable objects. This shows that $\calC$ is locally finite and completes the proof of Theorem~\ref{Dynkin-theorem}.

It is not part of the proof, but we may note that the dimensions of Hom spaces between $M$ and other objects in $\calC$ are now completely determined as the unique function $f_M$ which is additive everywhere except on the meshes terminating at $M$ and $M[1]$, where it has defect $d$.

\section{Proof of Theorem~\ref{extended-Dynkin-theorem}}

We suppose that  $\Gamma_1$ and  $\Gamma_2$ are stable components of the Auslander-Reiten quiver of $\calC$ and that $\Gamma_1$ has an extended Dynkin diagram as its tree class.
\medskip

If $\Gamma_2$ is a shift of $\Gamma_1$ then it has the same tree class, which is an extended Dynkin diagram, and we are done. Thus we may suppose that $\Gamma_2$ is not a shift of $\Gamma_1$ and it follows, by Lemma~\ref{f-and-g-lemma} that for each $M$ in $\Gamma_2$ the functions $f_M$ and $g_M$ are additive on $\Gamma_1$. They are also non-negative, and hence by Lemma~\ref{additive-function-existence-lemma} they are always periodic, of bounded period. 
\medskip

Since $\Hom(\tau^n M,X)\cong \Hom(M,\tau^{-n}X)$ it now follows that all functions $f_X$ and $g_X$ are periodic on $\Gamma_2$ when $X\in\Gamma_1$, of bounded period. We can find a non-zero such function. This enables us to put a non-negative additive function on the tree of $\Gamma_2$, which implies that it must be one of the trees listed, by \cite{HPR}.

\end{document}